\documentclass[11pt,reqno]{amsart}
\usepackage{hyperref}\hypersetup{colorlinks=true, citecolor=blue}
\usepackage{graphicx}
\usepackage{xfrac}
\usepackage{amsfonts,amsmath,amssymb,epsfig,mathrsfs,bm}
\usepackage{ifthen,color,mathtools}
\usepackage{epic,eepic}
\usepackage{esint,yfonts,stmaryrd}
\usepackage{a4wide,accents}

\newtheorem{theorem}{Theorem}[section]
\newtheorem{lemma}[theorem]{Lemma}
\newtheorem{proposition}[theorem]{Proposition}
\newtheorem{corollary}[theorem]{Corollary}

\theoremstyle{definition}

\theoremstyle{remark}
\newtheorem{remark}[theorem]{Remark}
\theoremstyle{remark}

\numberwithin{equation}{section}

\newcommand{\R}{\mathbb{R}}

\newcommand{\diw}{\mathrm{div}\,}
\newcommand{\Ln}{\mathcal{L}^n}

\newcommand{\LL}{\mathbb{L}}
\newcommand{\HH}{{\mathcal H}^{n-1}}
\newcommand{\EEE}{\mathscr{E}}
\newcommand{\HHH}{\mathscr{H}}
\newcommand{\Om}{\Omega}
\newcommand{\zz}{\underline{0}}
\newcommand\Sing{\textup{Sing}}
\newcommand\Reg{\textup{Reg}}
\newcommand{\cL}{\mathcal{L}}

\renewcommand{\d}{\,{\rm d}}

\title[The classical obstacle problem with H\"older continuous coefficients]{The classical obstacle problem with 
H\"older continuous coefficients}

\author[G. Andreucci]{Giovanna Andreucci}
\address{Universit\`a di Roma ``La Sapienza''}
\curraddr{P.le Aldo Moro 5,  Roma (Italy)}
\email{giovanna.andreucci@uniroma1.it}
\author[M.~Focardi]{Matteo Focardi}
\address{DiMaI, Universit\`a degli Studi di Firenze}
\curraddr{Viale Morgagni 67/A, 50134 Firenze (Italy)}
\email{matteo.focardi@unifi.it}
\thanks{M.~F. is a member of the Gruppo Nazionale per
l'Analisi Matematica, la Probabilit\`a e le loro Applicazioni (GNAMPA)
of the Istituto Nazionale di Alta Matematica (INdAM)}



\dedicatory{{This paper is dedicated to Prof. Emmanuele Di Benedetto\\ 
for having taught us, with his body of work, the taste for hard analysis problems.}}
\date{}
\begin{document}
\begin{abstract}
Weiss' and Monneau's type quasi-monotonicity formulas are established for quadratic energies 
having matrix of coefficients which are 
Dini, double-Dini continuous, respectively. 
Free boundary regularity for the corresponding classical obstacle problems under H\"older 
continuity assumptions is then deduced. 
\end{abstract}

\maketitle

%
%
\section{Introduction}\label{s:intro}

In the last years several contributions have been devoted to the extension of the regularity theory for obstacle type problems to the case in which the involved linear elliptic operator in divergence form has coefficients with low regularity \cite{Blank01,GSVG14, BlankHao15,BlankHao15bis,FoGeSp15,KRS16, KRS17, RS17,Ger17,FoGerSp20,JPSm}. 
The aim of this note is to make another step in that direction for the classical obstacle problem by establishing Weiss' and Monneau's quasi-monotonicity formulas for quadratic forms having matrix of coefficients which are Dini, double-Dini continuous, respectively (for the sake of simplicity the lower order terms have the same regularity). Such results are instrumental to pursue the variational approach to establish the smoothness of the corresponding free boundaries. To keep the presentation as simple as possible the related free boundary analysis is performed 
only in the H\"older continuous case.

More precisely, in what follows we consider the functional $\mathcal{E}:W^{1,2}(\Omega)\to\R$ given by
 \begin{equation}\label{e:enrg}
 \mathcal{E}(v):=\int_\Omega \big(\langle \mathbb{A}(x)\nabla v(x),\nabla v(x)\rangle +2h(x)v(x)\big)\,\d x,
\end{equation}
and study regularity issues related to its unique minimizer $w$ on the set 
\[
\mathcal{K}_{\psi,g} := \big\{v\in W^{1,2}(\Omega):\,v\geq \psi\;\;\Ln\text{-a.e. on } \Omega,\,
\textup{Tr}(v)=g \,\text{ on } \partial\Omega\big\}\,.
\]
Existence and uniqueness of the minimizer is a standard issue for the problem under the ensuing hypotheses.
For a bounded open set $\Omega\subset \R^n$, $n\geq 2$,
we consider in what follows functions $\psi\in C^{1,1}_{\mathrm{loc}}(\Omega)$ and 
$g\in H^{\sfrac12}(\partial\Omega)$, such that $\psi\leq g$ $\mathcal{H}^{n-1}$-a.e on $\partial\Omega$, 
a matrix-valued field $\mathbb{A}:\Omega\to \R^{n\times n}$ and a function $f:\Omega\to \R$ satisfying:
\begin{itemize}
\item[(H1)] $\mathbb{A}(x)=\left(a_{ij}(x)\right)_{i,j=1,\dots,n}\in L^\infty(\Omega,\R^{n\times n})$ is symmetric and coercive, that is $a_{ij}=a_{ji}$ $\cL^n$-a.e. in $\Omega$ for all $i,\,j\in\{1,\ldots,n\}$,
  and for some $\Lambda\geq 1$ 
  \begin{equation}\label{A coerc cont}
   \Lambda^{-1}|\xi|^2\leq \langle \mathbb{A}(x)\xi, \xi\rangle \leq \Lambda|\xi|^2
  \end{equation}
 for $\cL^n$-a.e. $x\in\Omega$, and for all $\xi\in \R^n$; 
\item[(H2)] $f\in L^\infty(\Omega)$, $f:=h-\diw(\mathbb{A}\nabla\psi)>c_0$ $\cL^n$-a.e.
 on $\Omega$, for some $c_0>0$;
 \item[(H3)]  $\mathbb{A}$ and $f$ are $\alpha$-H\"older continuous, for some $\alpha\in(0,1]$,
\end{itemize}
(cf. Section~\ref{s:prel} for the notation).
Under assumptions (H1)-(H3) we establish a stratification result for the free boundary of solutions (see also Remark~\ref{r:extensions} for further possible extensions). 
\begin{theorem}\label{t:linear}
Assume (H1)-(H3), and let $w$ be the (unique) minimizer 
of $\mathcal{E}$ in \eqref{e:enrg} on $\mathcal{K}_{\psi,g}$.

Then, $w$ is $C^{1,\alpha}_{\mathrm{loc}}(\Omega)$, and the free boundary 
decomposes as $\partial \{w = \psi\} \cap \Omega = \Reg(w) \cup \Sing(w)$, where 
$\Reg(w)$ and $\Sing(w)$ are called its regular and singular part, respectively. 
Moreover, $\Reg(w) \cap \Sing(w) = \emptyset$ and
\begin{itemize}
\item[(i)] 
$\Reg(w)$ is relatively open in $\partial \{w = \psi\}$
and, for every point $x_0 \in \Reg(w)$, there exist $r=r(x_0)>0$ such that $\partial \{w = \psi\}\cap B_r(x)$ 
is a $C^{1,\beta}$ $(n-1)$-dimensional manifold, 
for some exponent $\beta\in (0,1)$. 
 

\item[(ii)] 
$\Sing(w) = \cup_{k=0}^{n-1} S_k$, with $S_k$ contained in the union of at most countably many submanifolds of dimension $k$ and class $C^1$.
\end{itemize}
\end{theorem}
In the case of smooth matrix fields, Theorem~\ref{t:linear} collects 
the fundamental contributions of Caffarelli to the classical obstacle problem (cf. for instance \cite{Caf77, Caf80, Caf98-Fermi, Caf98} 
and the books \cite{KS, CS, PSU} for more details and references 
also on related issues). 

In the last years Theorem~\ref{t:linear} has been extended to the case in which $\mathbb{A}$ either is
Lipschitz continuous in \cite{FoGeSp15} or belongs to a fractional Sobolev space $W^{1+s,p}$  in \cite{Ger17}, with $sp>1$ and 
$p\geq n\wedge\frac{n^2}{n(1+s)-1}$, or belongs to the Sobolev space $W^{1,p}$ with $p>n$ in \cite{FoGerSp20}. 
We point out that in all those cases the involved matrix fields $\mathbb{A}$ turns out to be H\"older continuous in view of Sobolev type embeddings. Applications of the results in \cite{FoGeSp15,FoGerSp20} to the study of the obstacle problem for a large class of nonlinear energies were then given in \cite{FoGerSp17}. 
In addition, counterexamples to smoothness of the free boundary are shown in case of the Laplacian 
if $f$ is not Dini continuous in \cite{Blank01}, and in case $f$ is constant and $\mathbb {A}$ is 
$\mathrm{VMO}$ in \cite[Remark~4.2]{BlankHao15bis}.
In this respect, if both the coefficients $\mathbb{A}$ and $f$ are assumed to be in $\mathrm{VMO}$, 
Reifenberg vanishing flatness of $\Reg(u)$ 
has been established in \cite{BlankHao15bis}.

The papers \cite{FoGeSp15,Ger17,FoGerSp20} follow the variational approach to free boundary analysis 
remarkably developed by Weiss \cite{Weiss} and by Monneau \cite{Monneau03}. Central to this method
are quasi-monotonicity formulas for suitable quantities, named in literature as Weiss' and Monneau's 
quasi-monotonicity formulas.
The original papers by Weiss and Monneau are based on some key integral identities that hold true in the case of the Laplacian. The search for generalizations need several new insights and technical tools. Let us briefly recall the key ideas introduced in \cite{FoGeSp15,Ger17,FoGerSp20}. 
An extension of the Rellich and Ne\v{c}as' inequality due to Payne and Weinberger (cf. \cite{Kukavica}) is employed 
in the papers \cite{FoGeSp15,Ger17}. On a technical side, the matrix field $\mathbb{A}$ is differentiated, therefore enough smoothness of $\mathbb{A}$ is needed. 
Instead, a different approach, leading to a weakening of the regularity assumptions on $\mathbb{A}$, has been pursued in \cite{FoGerSp20}
following an observation in \cite[Remark 4.9]{FoGeSp15}. 
 The main difference in \cite{FoGerSp20} with respect to \cite{FoGeSp15,Ger17} concerns the monotone quantity itself: 
rather than considering the natural quadratic energy associated to the obstacle problem under study, a related constant coefficient quadratic form has been used. 
Higher order regularity of solutions, namely $W^{2,p}_{\mathrm{loc}}$, 
in combination with their well-known quadratic growth from free boundary points then allow to conclude. 
In this note we push forward the ideas in \cite{FoGerSp20} and extend the conclusions there to H\"older continuous matrix fields $\mathbb{A}$ for which the above mentioned $W^{2,p}_{\mathrm{loc}}$ regularity of solutions is not guaranteed.
More precisely, we prove Weiss' formula for Dini continuous matrix fields $\mathbb{A}$ and Monneau's formula for double-Dini continuous ones. For simplicity we do not distinguish the degrees of regularity of $\mathbb{A}$ and that of $f$. Let us stress that in the case of the Laplacian, low regularity assumptions on $f$ have already been considered: Dini continuity has been deeply investigated in \cite{Blank01}, and Dini continuity in $L^p$ sense, $p$ suitable, in \cite{Monneau09} (cf. also \cite{BlankHao15bis} for both $\mathbb{A}$ and $f$ in $\mathrm{VMO}$; \cite{FoGerSp20} for Theorem~\ref{t:linear} if $\mathbb{A}$ is $W^{1,p}$, $p>n$, and 
$f$ double-Dini continuous). 
A potential extension of the related free boundary analysis contained in 
Theorem~\ref{t:linear} is discussed in Remark~\ref{r:extensions}
(cf. also Remark~\ref{r:Weiss generalized}).
Note that if $\alpha=1$ we recover the results in \cite{FoGeSp15} with, partially, a different proof.

We close this introduction briefly resuming the structure of the paper: standard preliminaries for the classical obstacle problem 
are collected in Section~\ref{s:prel}. 
The mentioned generalizations of Weiss' and Monneau's quasi-monotonicity formulas are dealt with in Section~\ref{s:q-mon formula}, 
finally Section~\ref{s:applications} contains the applications to the free boundary stratification problem.

%
%
\section{Preliminaries}\label{s:prel}

%
\subsection{Notation}
Throughout the whole paper, the inner product in $\R^n$ is denoted by $\langle\cdot,\cdot\rangle$, and the induced norm by $|\cdot|$.
To distinguish it from the norm in $\R^{n\times n}$ we use the symbol $\|\cdot\|$ for the latter.
We use standard notations for Lebesgue, Sobolev and H\"older spaces, quoting the necessary results when needed.

\subsection{Reduction to the zero obstacle case}

We first reduce ourselves to the zero obstacle problem. 
Let $w$ be the unique minimizer of $\mathcal{E}$ over $\mathcal{K}_{\psi,g}$, 
and define $u:=w-\psi$. Then, $u$ is the unique minimizer of 
 \begin{equation}\label{e:enrg2}
 \mathscr{E}(v):=\int_\Omega \big(\langle \mathbb{A}(x)\nabla v(x),\nabla v(x)\rangle +2f(x)v(x)\big)\,\d x,
\end{equation}
over $\mathbb{K}_{\psi,g}:=\mathcal{K}_{0,\psi-g}$,
where $f=h-\diw(\mathbb{A}\nabla\psi)$. Clearly, $\partial\{w=\psi\}\cap\Omega=\partial\{u=0\}\cap\Omega$, therefore 
we shall establish all the results in Theorem~\ref{t:linear} for $u$.
Notice that assumption (H2) is formulated exactly in terms of $f$,
and that, for the zero obstacle problem, i.e.~$\psi=0$ the positivity assumption on $f$ in (H2) involves only the lower order term $h$ 
in the integrand and not the matrix field $\mathbb{A}$. 
It is also clear that given (H3), the $L^\infty$ assumptions in (H1) 
and (H2) are redundant and the inequalities there hold in a pointwise sense. Despite this, we state (H1) and (H2) as in the Introduction because for some of the ensuing results we shall not need (H3) and we will sometimes substitute 
it with a weaker version (cf. (H4) and (H5) in what follows).

The first step is to prove that $u$ satisfies a PDE both in the distributional sense and $\cL^n$-a.e. on $\Omega$.
The next result has been established in \cite[Proposition 3.2]{FoGerSp17} more generally for variational inequalities
inspired by some arguments in \cite{Weiss} (see also \cite{Fu90,FuMi00}). 
The leading vector field in \cite[Proposition 3.2]{FoGerSp17} is assumed to be $C^{1,1}_{\mathrm{loc}}$ in the full set of variables for the sake of existence of solutions in that general setting. 
On the contrary, such a regularity in the space variable is never used to deduce the conclusion of \cite[Proposition~3.2]{FoGerSp17}.
We provide below the proof for the readers' convenience.
\begin{proposition}\label{p:PDE u}
Assume (H1) and (H2). Let $u$ be the minimizer of $\mathscr{E}$ on $\mathbb{K}_{\psi,g}$. Then, there exists a function 
$\zeta\in L^\infty(\Omega)$ such that 
 \begin{equation}\label{e:PDE_u}
  \mathrm{div}(\mathbb{A}\nabla u)=f-\zeta 
 \end{equation}
$\Ln$-a.e. on $\Omega$ and in $\mathcal{D}'(\Omega)$, with 
\begin{equation}\label{e:zeta}
0\leq\zeta\leq f\chi_{\{u=0\}}\,.
\end{equation}
In particular, $u\in C^{0,\beta}_{\mathrm{loc}}(\Omega)$ for 
some $\beta\in(0,1)$. 
\end{proposition}
\begin{proof}
	Let $ \varphi \in C^{\infty}_{c}(\Omega) $, and for all 
	$\varepsilon>0$ define $ v_{\varepsilon} \coloneqq (u + \varepsilon \varphi) \vee 0$. Note that $ v_{\varepsilon} $ belongs to the set $\mathbb{K}_{\psi,g}$.
	If $ \varphi \geq 0 $ then $ v_{\varepsilon}= u + \varepsilon \varphi $ and, since $ u $ is the minimizer we  have
	%
	\begin{align*}
	-\frac{\varepsilon}{2} \int_{\Omega} \langle \mathbb{A}(x)\nabla \varphi ,\nabla \varphi \rangle\d x \leq \int_{\Omega} \langle \mathbb{A}(x)\nabla u, \nabla \varphi\rangle \d x + \int_{\Omega}f(x) \varphi \d x\,.
	\end{align*}
In turn, for $ \varepsilon \rightarrow 0 $ this implies that
	\begin{align*}
	0 \leq \int_{\Omega} \langle \mathbb{A}(x)\nabla u, \nabla \varphi)\rangle \d x + \int_{\Omega}f(x) \varphi \d x \qquad \forall \varphi 	\in C^{\infty}_{c}(\Omega),\; \varphi \geq 0.
	\end{align*}
	The last inequality yields that the distribution $\mu :=-\mathrm{div}\big( \mathbb{A}(\cdot) \nabla u\big)+f(\cdot)\mathcal{L}^{n}\llcorner \Omega$ is a non-negative Radon measure on $\Omega$.
	
Let now $ \varphi $ and $ v_{\varepsilon} $ be as above without any assumption on the sign of $ \varphi $, and set 
\[
\Omega_{\varepsilon}:=\lbrace x\in \Omega:\, u+\varepsilon \varphi < 0 \rbrace\,.
\] 
It is clear that 
$v_{\varepsilon}=0$ and $\varphi<0$ on $ \Omega_{\varepsilon}$, and that $v_{\varepsilon}=u+\varepsilon \varphi$ on 
$\Omega\setminus \Omega_{\varepsilon}$. 
Using again that $u$ is minimizing, that $f\geq c_0>0$ on 
$\Omega$ by (H2), and that $ v_{\varepsilon} \in \mathbb{K}_{\psi,g} $ we obtain
\begin{align*}
0\leq&\;  \mathscr{E}(v_\varepsilon)-\mathscr{E}(u)
=- \int_{\Omega_{\varepsilon}} \langle \mathbb{A}\nabla u ,\nabla u \rangle \d x 	+2\varepsilon\int_{\Omega\setminus\Omega_\varepsilon}\langle \mathbb{A}\nabla u,\nabla \varphi \rangle \d x
\\&+\varepsilon^{2} \int_{\Omega\setminus \Omega_{\varepsilon}}\langle \mathbb{A}\nabla \varphi ,\nabla \varphi\rangle\d x
+2\int_{\Omega\setminus\Omega_{\varepsilon}}f (u+\varepsilon\varphi) \d x-2\int_{\Omega} f u\d x\\
\leq\;&
2\varepsilon\int_{\Omega}\langle \mathbb{A}\nabla u,\nabla \varphi \rangle \d x
+2\varepsilon\int_{\Omega} f\varphi \d x
-2\varepsilon\int_{\Omega_{\varepsilon}}\langle \mathbb{A}\nabla u,\nabla \varphi \rangle \d x\\
&+\varepsilon^{2} \int_{\Omega\setminus \Omega_{\varepsilon}}\langle \mathbb{A}\nabla \varphi ,\nabla \varphi\rangle\d x -\int_{\Omega_{\varepsilon}}2f (u+\varepsilon\varphi) \d x\,.
\end{align*}
In the last inequality we have dropped the first term on the first line as it is negative (cf. (H1)).
Thus, on account of the definition of $ \mu $, the last formula rewrites as
	\begin{equation}\label{interg}
	\int_{\Omega}\varphi \d\mu\geq I_{1}(\varepsilon)+I_{2}(\varepsilon)+I_{3}(\varepsilon)
	\end{equation}
	where
	\begin{align*}
I_{1}(\varepsilon)=\int_{\Omega_{\varepsilon}} \langle \mathbb{A}(x)\nabla u, \nabla \varphi\rangle \d x\,,\quad
 I_{2}(\varepsilon)=-\frac{\varepsilon}{2}\int_{\Omega\setminus \Omega_{\varepsilon}} \langle \mathbb{A}\nabla \varphi ,\nabla \varphi\rangle\d x\,,\quad
 I_{3}(\varepsilon) =\frac1\varepsilon  
\int_{\Omega_{\varepsilon}} f(u+\varepsilon\varphi)\d x\,.
\end{align*}
First, we note that the very definition of $\Omega_\varepsilon$ yields that 	
\[
	\mathcal{L}^{n}\left( \big(\lbrace u=0\rbrace\cap\lbrace \varphi<0\rbrace\big)\setminus \Omega_{\varepsilon}\right)
	=\mathcal{L}^{n}\left( \Omega_{\varepsilon} \setminus \big( \lbrace 0\leq u\leq \varepsilon \| \varphi \|_{L^{\infty}(\Omega)}\rbrace \cap \lbrace \varphi<0\rbrace\big)\right)=0,
\]
and thus we may conclude that $\chi_{\Omega_{\varepsilon}} \rightarrow \chi_{\lbrace u=0\rbrace\cap\lbrace \varphi<0\rbrace}$ 
in $L^{1}(\Omega)$.
Hence, the dominated convergence theorem yields
\begin{equation}\label{I1}
\lim_{\varepsilon\rightarrow 0^{+}} I_{1}(\varepsilon)=\int_{{\lbrace u=0\rbrace\cap\lbrace \varphi<0\rbrace}} \langle \mathbb{A}\nabla u,\nabla \varphi\rangle \d x=0\,,
\end{equation}
by locality of the weak gradient. 

Instead, to estimate $I_2(\varepsilon)$, we use (H1) as follows
\[
|I_2(\varepsilon)|=\frac\varepsilon2\int_{\Omega\setminus \Omega_{\varepsilon}} \langle \mathbb{A}\nabla \varphi ,\nabla \varphi\rangle\d x
\leq \frac\varepsilon2\Lambda
\int_{\Omega} |\nabla \varphi|^{2} \d x\,,
\]
so that
	\begin{equation}\label{I3}
	\lim_{\varepsilon\rightarrow 0^{+}}I_2(\varepsilon)= 0\,.
	\end{equation}
	Lastly, we deal with  $I_3(\varepsilon)$: 
note that 
$u\geq 0$ and $f\geq c_0>0$ on $\Omega$ by (H2), thus we have
\[
\int_{\Omega_{\varepsilon}} f (u+\varepsilon\varphi)\d x 
\geq\varepsilon \int_{\Omega_{\varepsilon}} f \varphi \d x \,.
\]
 Hence, it is true that 
	\begin{equation}\label{I2}
	\liminf_{\varepsilon\rightarrow 0^{+}} I_{3}(\varepsilon)
	\geq\int_{{\lbrace u=0\rbrace\cap\lbrace \varphi<0\rbrace}}f \varphi \d x\,.
	\end{equation}
	Therefore, \eqref{interg}-\eqref{I2} yield that
	\[
	\int_{\Omega} \varphi \d \mu \geq 
	\int_{{\lbrace u=0\rbrace\cap\lbrace \varphi<0\rbrace}}f \varphi \d x. 
	\]
By repeating the same argument with $-\varphi$ we obtain the inequality
	\[
	\int_{\Omega} \varphi \d \mu \leq 
	\int_{{\lbrace u=0\rbrace\cap\lbrace \varphi>0\rbrace}}f \varphi \d x\,.
	\]
	Hence, by approximation for every $\varphi\in C^{0}_{0}(\Omega)$ it holds
\begin{equation}\label{e:mu finale}
 \int_{{\lbrace u=0\rbrace\cap\lbrace \varphi<0\rbrace}}f \varphi \d x
 \leq\int_{\Omega} \varphi \d\mu \leq
 \int_{{\lbrace u=0\rbrace\cap\lbrace \varphi>0\rbrace}}f \varphi \d x\,.
\end{equation}
From this we infer that $\mu \ll \mathcal{L}^{n}\llcorner \Omega$, and thus $\mu=\zeta \mathcal{L}^{n}\llcorner \Omega$ with $\zeta\in L^{1}(\Omega)$. In conclusion, plugging this piece of information in \eqref{e:mu finale} we conclude that 
$0\leq \zeta \leq f \chi_{\lbrace u=0 \rbrace}$ $\mathcal{L}^{n}$-a.e. on $\Omega$.

The H\"older continuity of $u$ follows from nowadays standard 
elliptic regularity \cite[Theorem~8.22]{GT}.
\end{proof}
\begin{remark}\label{r:Blank}
In case $\mathbb{A}$ is more regular, namely $W^{1,p}(\Omega;\R^{n\times n})$ for $p>n$, the minimizer $u$ turns out to be $W^{2,p}_{\mathrm{loc}}(\Omega)$  (cf. \cite[Proposition~3.1]{FoGerSp20}). In turn, from this one can easily prove that \eqref{e:PDE_u} rewrites as
\begin{equation}\label{e:PDE_u refined}
\mathrm{div}(\mathbb{A}\nabla u)=f \chi_{\{u>0\}}
 \end{equation}
 $\Ln$-a.e. on $\Omega$ and in $\mathcal{D}'(\Omega)$
(cf.  \cite[Corollary~3.5]{FoGerSp17}, \cite[Proposition 3.1]{FoGerSp20}).

Here we limit ourselves to notice that, standing assumptions (H1) and (H2), by uniqueness of $u$ and \eqref{e:PDE_u}, $\zeta$ coincides with $f$ on the interior $\{u=0\}^\circ$ of $\{u=0\}$, so that we may conclude the refined inequalities
\[
f\chi_{\{u=0\}^\circ}\leq\zeta\leq f\chi_{\{u=0\}}\,.
\]
This remark will be crucial in what follows in order to apply 
\cite[Theorems~3.9]{BlankHao15}.
In particular, we infer from this  the quadratic 
non-degeneracy of the solution from free boundary points. 
Instead, a parabolic bound from above from free boundary points follows directly from \eqref{e:PDE_u} thanks to \cite[Theorems~3.1]{BlankHao15}.
We apply these properties to the blow up analysis of Section~\ref{s:q-mon formula} (cf. Proposition~\ref{p:quadratic growth detachment}).
 
Furthermore, observe that $\cL^n(\Gamma_u)=0$ in view of \cite[Corollary 3.10]{BlankHao15}, 
so that equation \eqref{e:PDE_u refined} is actually true $\Ln$-a.e. on $\Omega$ and in $\mathcal{D}'(\Omega)$. 
Despite this, we stress that we do not need such a piece of information to establish the 
quasi-monotonicity formulas in Section~\ref{s:q-mon formula}: equation \eqref{e:PDE_u}, 
but not \eqref{e:PDE_u refined}, will be used in the proofs of Theorems~\ref{t:Weiss} and \ref{t:Monneau}. 
\end{remark}
We establish next two useful corollaries of Proposition~\ref{p:PDE u}. The first is simply a suitable version of Caccioppoli's inequality.
\begin{corollary}\label{c:Caccioppoli}
Assume (H1) and (H2). Let $u$ be the minimizer of $\mathscr{E}$ on $\mathbb{K}_{\psi,g}$. Then, there exists a constant 
$C=C(n,\Lambda)>0$ such that for every $x_0\in\Omega$ and for every $r\in(0,\frac14\mathrm{dist}(x_0,\partial\Omega))$
\begin{equation}\label{e:Caccioppoli}
\int_{B_r(x_0)}|\nabla u|^2\d x\leq\frac{C}{r^2}
\int_{B_{2r}(x_0)}u^2\d x+C\|f\|_{L^\infty(B_{2r}(x_0))}^2r^{n+2}\,.
\end{equation}
\end{corollary}
Instead, the second corollary is an integration by parts formula which 
will be employed in the proof of the Monneau's quasi-monotonicity formula. 
\begin{corollary}\label{c:integration by parts}
Assume (H1) and (H2). Let $u$ be the minimizer of $\mathscr{E}$ on $\mathbb{K}_{\psi,g}$. Then,
for every $x_0\in\Omega$ and for $\mathcal{L}^1$-a.e. $r\in(0,\mathrm{dist}(x_0,\partial\Omega))$
\begin{equation}\label{e:singular variation}
\int_{B_r(x_0)}\langle\mathbb{A}\nabla u,\nabla\varphi\rangle\d x+
\int_{B_r(x_0)}(f-\zeta)\varphi\d x=
\int_{\partial B_r(x_0)}\langle\mathbb{A}\nabla u,\nu\rangle\varphi\d\mathcal{H}^{n-1}
\end{equation}
for every $\varphi\in W^{1,2}(\Omega)$, where $\nu(x):=\frac{x-x_0}{|x-x_0|}$. 
\end{corollary}
\begin{proof}
For $x_0\in\Omega$, $\varphi\in W^{1,2}(\Omega)$, $r\in(0,\mathrm{dist}(x_0,\partial\Omega))$, 
and $s<r$ let 
\[
\psi_s(x):=1\wedge\frac1{r-s}{\mathrm{dist}(x,\partial B_r(x_0))}
\] 
if $x\in B_r(x_0)$ and $\psi_s(x):=0$ otherwise in $\Omega$. Then $\psi_s\varphi\in W^{1,2}_0(\Omega)$ 
so that using it as test function in \eqref{e:PDE_u} we get
\[
\int_{B_r(x_0)}\langle\mathbb{A}\nabla u,\nabla\varphi\rangle\psi_s\d x+
\int_{B_r(x_0)}(f-\zeta)\varphi\psi_s\d x=
-\int_{B_r(x_0)}\langle\mathbb{A}\nabla u,\nabla\psi_s\rangle\varphi\d x:=I_s\,.
\]
It is clear that $\psi_s\to \chi_{B_r(x_0)}$ in $L^p(\Omega)$ for every $p\in[1,\infty)$ as $s\uparrow r$. Moreover, being 
$0\leq\psi_s\leq 1$, the Lebesgue dominated convergence theorem implies that the left hand side converges to the left hand side in \eqref{e:singular variation} as $s\uparrow r$. 
Computing explicitly the gradient of $\psi_s$, we evaluate $I_s$ as follows:
\begin{align*}
I_s=\frac1{r-s}\int_{B_r(x_0)\setminus B_s(x_0)}\langle\mathbb{A}\nabla u,
{\textstyle{\frac{x-x_0}{|x-x_0|}}}\rangle\varphi\d x
=\frac1{r-s}\int_s^r\d t\int_{\partial B_t(x_0)}\langle\mathbb{A}\nabla u,\nu\rangle\varphi\d\mathcal{H}^{n-1}\,,
\end{align*}
where in the second equality we have used the coarea formula. The conclusion then follows at once 
(cf. for instance \cite[Sections~3.4.3 and 3.4.4]{EG}). 
\end{proof}

We recall next the standard notations for the coincidence set and 
for the free boundary 
 \begin{equation}
  \Lambda_u:=\{x\in\Omega:\,u(x)=0\}\,,
  \qquad\Gamma_u :=\partial\Lambda_u\cap \Omega.
 \end{equation}
For any point $x_0\in \Gamma_u$, we introduce the family of rescaled functions
\begin{equation}\label{u_x_0 r}
u_{x_0,r}(x):=\frac{u(x_0+rx)}{r^2}
\end{equation}
for $x\in\frac 1r(\Omega-x_0)$. It is clear that 
changing variables in \eqref{e:enrg2} implies that $u_{x_0,r}$ minimizes 
\begin{equation}\label{e:enrg2 u_r}
\int_{\frac 1r(\Omega-x_0)}\big(\langle \mathbb{A}(x_0+r x)\nabla v(x),\nabla v(x)\rangle +2f(x_0+r x)v(x)\big)\,\d x,
\end{equation}
among all functions $v\geq0$ on $\frac 1r(\Omega-x_0)$, and with $v-u_{x_0,r}\in W^{1,2}_0(\frac 1r(\Omega-x_0))$. 
Note that $\mathbb{A}(x_0+r \cdot)$ and $f(x_0+r \cdot)$ satisfy
(H1)-(H2) uniformly in $x_0$ and $r$, i.e. the ellipticity constants of
$\mathbb{A}(x_0+r \cdot)$ are the same of those of 
$\mathbb{A}$, the $L^\infty$ bound and the lower bound for $f(x_0+r \cdot)$ 
are the same of those of $f$. Therefore, similarly 
to \eqref{e:PDE_u}, from \eqref{e:enrg2 u_r} we infer that 
\begin{equation}\label{e:PDE u_r}
\mathrm{div}(\mathbb{A}(x_0+rx)\nabla u_{x_0,r})=f(x_0+rx)-\zeta(x_0+r x)
\end{equation}
$\cL^n$-a.e. on $\frac 1r(\Omega-x_0)$ and in $\mathcal{D}'(\frac 1r(\Omega-x_0))$, 
where $\zeta$ is the function in Proposition~\ref{p:PDE u} (cf \cite[Section 8]{GT}).

The first properties we recall on the family $(u_{x_0,r})_{r}$
follows from the fundamental quadratic growth and quadratic detachment of the solution from free boundary points established in \cite[Theorems~3.1, 3.9]{BlankHao15} (cf. the discussion in Remark~\ref{r:Blank}, see also \cite[Theorems~2.3, 2.4]{BlankHao15bis}).
We summarize the needed properties in the ensuing statement which is key for our approach. 
We remark that the estimates are uniform in $x_0$ and $r$.
\begin{proposition}\label{p:quadratic growth detachment}
Assume (H1) and (H2). Let $u$ be the minimizer of $\mathscr{E}$ over $\mathbb{K}_{\psi,g}$. 
There exists a constant $\vartheta=\vartheta(n,\Lambda,c_0,\|f\|_{L^\infty})>0$, such that for every $x_0 \in \Gamma_u$ and for every $r\in(0,\frac12\textup{dist}(x_0,\partial\Omega))$, it holds
\begin{equation}\label{e:quadratic detachment}
\sup_{\partial B_1} u_{x_0,r} \geq \vartheta\,.
\end{equation} 
 Moreover, for every $R>0$ and for all compact sets $K\subset \Omega$ there exists a constant $C=C(n,\Lambda,c_0,\|f\|_{L^\infty},R,K)>0$ such that 
\begin{equation}\label{e:quadratic growth0}
 \|u_{x_0,r}\|_{L^\infty(B_R)}\leq C\,,
\end{equation}
for all $x_0\in \Gamma_u\cap K$, and for all $r\in \big(0, \frac{1}{4R}\mathrm{dist}(K, \partial\Omega)\big)$.
\end{proposition}
\begin{remark}\label{r:Caccioppoli}
Notice that in view of the estimate in \eqref{e:Caccioppoli} in 
Corollary \ref{c:Caccioppoli} and the estimate in 
\eqref{e:quadratic growth0}, we may infer that for all $x_0\in \Gamma_u\cap K$, and for all $r\in \big(0, \frac{1}{4R}\mathrm{dist}(K, \partial\Omega)\big)$
\begin{equation}\label{e:quadratic growth}
\|u_{x_0,r}\|_{L^\infty(B_R)}+
\|\nabla u_{x_0,r}\|_{L^2(B_R,\R^n)}\leq C\,.
\end{equation}
\end{remark}
In addition, if assumption (H3) holds, then the H\"older seminorms of 
$\mathbb{A}(x_0+r \cdot)$ have uniform bounds with respect to $x_0$ and $r$. 
Thus, uniform Schauder estimates are deduced thanks to 
\cite[Corollary 8.36]{GT}.
Then the existence up to subsequences of $C^{1,\beta}$-limits, $\beta<\alpha$, as $r\downarrow 0$ 
of the family $(u_{x_0,r})_r$ is standard.
\begin{proposition}\label{p:prop u_r limitata}
Assume (H1)-(H3). Let $u$ be the minimizer of $\mathscr{E}$ over $\mathbb{K}_{\psi,g}$, and $K\subset\Omega$ a compact set.
Then for every $x_0\in K\cap\Gamma_u$ and for every $R>0$
 there exists a constant $C=C(n,\Lambda,\|f\|_{L^\infty},\|\mathbb{A}\|_{C^{0,\alpha}},R,K)>0$
 such that, for every $r\in (0,\frac{1}{4R}\mathrm{dist}(K, \partial\Omega))$
 \begin{equation}\label{e:u_r limitata C1alfa}
  \|u_{x_0,r}\|_{C^{1,\alpha}(B_R)}\leq C.
 \end{equation}
In particular, $(u_{x_0,r})_r$ is relatively compact in $C^{1,\beta}_{\mathrm{loc}}(A)$, for all $\beta\in(0, \alpha)$, and for every open set $A\subset\hskip-0.125cm\subset\R^n$.
\end{proposition}
The functions arising in this limit process are called blow up limits.
In particular, the blow up limits are non-trivial, i.e. not identically zero, in view of Proposition~\ref{p:quadratic growth detachment}.
\begin{corollary}
Assume (H1)-(H3). Let $u$ be the minimizer of $\mathscr{E}$ over 
$\mathbb{K}_{\psi,g}$, and let $x_0\in \Gamma_u$. 
Then, for every sequence $r_k\downarrow 0$ there exists a subsequence $(r_{k_j})_j\subset (r_k)_k$ such 
that the rescaled functions $(u_{x_0,r_{k_j}})_j$ converge to a non trivial $C^{1,\alpha}(A)$ limit in $C^{1,\beta}_{\mathrm{loc}}(A)$, for all $\beta\in(0,\alpha)$, and for every open set $A\subset\hskip-0.125cm\subset\R^n$. 
\end{corollary}

\section{Quasi-monotonicity formulas}\label{s:q-mon formula}

In this section we establish Weiss' and Monneau's type quasi-monotonicity formulas. The monotone quantities we consider in Section~\ref{ss:normalization} are modeled upon the classical
Dirichlet energy as in \cite{FoGerSp20} under a pointwise normalization condition on the coefficients. 
We shall show in the subsequent Section~\ref{ss:WM qmon} 
how to reduce to that formulation in a pointwise way via a suitable change of variables following \cite{FoGeSp15}. 
The advantage of this approach is that the matrix field 
$\mathbb{A}$ is not differentiated in deriving the quasi-monotonicity formulas contrary to \cite{FoGeSp15, Ger17}.
In those papers, instead, the natural quadratic energy $\mathscr{E}$ associated to the obstacle problem under study
had been considered.  
The $W^{2,p}_{\mathrm{loc}}$ regularity of solutions and the quadratic growth from the free boundary were key properties to establish 
the quasi-monotonicity formulas in \cite{FoGerSp20} for the Dirichlet based quantities. 
The new contribution of the current paper is to avoid the use of the former piece of information, which is not guaranteed in our setting, thanks to an elementary energy comparison argument to prove
Weiss' formula (cf. Proposition~\ref{p:Weiss normalized}), and in turn thanks to the latter and to Proposition~\ref{p:PDE u} to prove Monneau's formula (cf. Proposition~\ref{p:Monneau normalized}).

\subsection{Weiss' and Monneau's quasi-monotonicity formulas under a normalization condition}\label{ss:normalization}

We establish Weiss' and Monneau's quasi-monotonicity formulas for the minimizer $u$ of $\EEE$ on 
$\mathbb{K}_{\psi,g}$ under the normalization condition
\begin{equation}\label{e:normalization}
x_0=\zz\in\Gamma_u,\qquad \mathbb{A}(x_0)=\mathrm{Id},\qquad f(x_0)=1\,.
\end{equation}
We will show in the next Section~\ref{ss:WM qmon} how to reduce to \eqref{e:normalization} in each free boundary point thanks to a change of variables.

Moreover, for the sake of possible future generalizations of Theorem~\ref{t:linear}, we establish  Weiss' and Monneau's formulas under Dini, double-Dini continuity assumptions, respectively (cf. Remark~\ref{r:Weiss generalized} for generalizations of Weiss' formula). 
Thus, we introduce some terminology. Given a uniformly continuous function $\zeta:\Omega\to\R^m$, $m\geq 1$, we consider a modulus of continuity of $\zeta$ namely an increasing function $\omega_\zeta:(0,\infty)\to(0,\infty)$ with 
$\lim_{t\to0}\omega_\zeta(t)=0$, and such that for all
$t\in(0,\mathrm{diam}\Omega)$
 \[
\sup_{x,y\in\Omega:\, |x-y|\leq t}\|\zeta(x)-\zeta(y)\|\leq
\omega_\zeta(t)\,.
\]
With a slight abuse of notation with respect to our conventions we have denoted by $\|\cdot\|$ 
the norm in $\R^m$ even in case $m=n$. In particular, $\zeta$ is said to be Dini continuous if for some 
$\omega_\zeta$ as above
 \begin{equation}\label{e:Dini continuity}
 \int_0^{\mathrm{diam}\Omega} \frac{\omega_\zeta(t)}{t}\d t < \infty\,,
\end{equation}
and double-Dini continuous if for some $a\geq 1$
\begin{equation}\label{H4}
\int_0^{\mathrm{diam}\Omega} \frac{\omega_\zeta(t)}t\,|\log t|^a\, \d t < \infty\,.
\end{equation}
Therefore, we introduce the following weaker assumptions, 
each substituting (H3) in some of the results contained in 
this section:
\begin{itemize}
\item[(H4)] $\mathbb{A}$ and $f$ are Dini continuous; 
\item[(H5)] $\mathbb{A}$ and $f$ are double-Dini continuous.  \end{itemize}
We point out that under condition (H4), \cite[Theorem~2.1 and Remark~2.2]{Li} imply that $u\in C^1(\Omega)$ (cf. also 
\cite{Lib99,DK17, DEK18} for more on Schauder estimates for linear elliptic PDEs with Dini type continuity conditions on the matrix field).

Given $v\in W^{1,2}(\Omega)$ we consider the Weiss' energy 
 \begin{equation}\label{e:Weiss energy}
  \Phi_{v}(r):=\frac{1}{r^{n+2}}\int_{B_r} \big(|\nabla v|^2 + 2\,v\big)\,\d x 
  - \frac{2}{r^{n+3}}\int_{\partial B_r} v^2\, d\mathcal{H}^{n-1}\,,
 \end{equation}
and prove its quasi-monotonicity for $v=u$ in case \eqref{e:normalization} is satisfied. 
Let us also introduce the bulk energy
\begin{equation}\label{e:E_u}
\EEE_v(r):=\int_{B_r}(|\nabla v|^2+2v)\d x
\end{equation}
and the boundary energy 
\begin{equation}\label{e:H_u}
\HHH_v(r):=\int_{\partial B_r}v^2\d \HH\,,
\end{equation}
so that
\[
\Phi_v(r)=\frac1{r^{n+2}}\EEE_v(r)-\frac2{r^{n+3}}\HHH_v(r)\,.
\]
In the rest of the section to ease the notation we write $u_r$ in place of $u_{\zz,r}$. 
Note then that 
\begin{equation}\label{e:estimate E_u-H_u}
\EEE_{u_r}(1)=\frac{\EEE_u(r)}{r^{n+2}}\leq C ,\qquad \HHH_{u_r}(1)=\frac{\HHH_u(r)}{r^{n+3}}\leq C\,,
\end{equation}
thanks to the bound \eqref{e:quadratic growth0} in Remark~\ref{r:Caccioppoli} (we stress that no continuity assumption on 
$\mathbb{A}$ or $f$ is needed).

We establish next two auxiliary results. The first is well-known, we prove it only for the readers' convenience
being the fundamental identity from which Weiss' quasi-monotonicity is inferred.
 \begin{lemma}
 Assume (H1), (H2) and (H4). 
Let $u$ be the minimizer of $\mathscr{E}$ over $\mathbb{K}_{\psi,g}$. If  $\zz\in\Omega$, for every $r\in(0,\mathrm{dist}(\zz,\partial\Omega))$ consider 
\begin{equation}\label{e:2-hom extension}
w_r(x):=|x|^2u\big(r{\textstyle{\frac{x}{|x|}}}\big)\,.
\end{equation}
Then, it is true that 
 \begin{equation}\label{e:Phi'}
  \Phi_{u}'(r)=\frac{n+2}r(\Phi_{w_r}(1)-\Phi_{u_r}(1))+
  \frac1r\int_{\partial B_1}(\langle\nabla u_r,\nu\rangle-2u_r)^2\d \mathcal{H}^{n-1}\,.
 \end{equation}
 \end{lemma}
\begin{proof}
By definition for every $r\in (0,\mathrm{dist}(\zz,\partial\Om))$ we have
\begin{equation}\label{e:phiprime}
\Phi^\prime_u(r)=\frac {\EEE_u^\prime(r)}{r^{n+2}}-(n+2)\frac{\EEE_u(r)}{r^{n+3}}
-2\frac{\HHH_u^\prime(r)}{r^{n+3}}+2(n+3)\frac {\HHH_u(r)}{r^{n+4}}.
\end{equation}
A direct computation then gives 
\begin{align}\label{e:E'}
\EEE_u'(r)
= \int_{\partial B_{r}} (| \nabla u|^{2} + 2u \,) \d \mathcal{H}^{n-1}=r^{n+1}\int_{\partial B_{1}} (| \nabla u_{r}|^{2} + 2u_{r} \,) \d \mathcal{H}^{n-1}.
\end{align}
By scaling it is easy to check that 
\begin{align}\label{e:H'}
\HHH_u'(r)&=
\dfrac{\d}{\d r}\left( \int_{\partial B_{1}} u^{2}(ry) r^{n-1}\d \mathcal{H}^{n-1} \right)\notag\\
&=\dfrac{n-1}{r}\HHH_u(r)+2r^{n+2}\int_{\partial B_{1}} u_r \langle\nabla u_r,y \rangle \d \mathcal{H}^{n-1}\,,
\end{align}
therefore
\begin{align}\label{e:H'+H}
-  \dfrac{2}{r^{n+3}}\HHH_u'(r)+& \dfrac{2(n+3)}{r^{n+4}}\HHH_u(r)
=\dfrac{8}{r^{n+4}}\HHH_u(r)-\dfrac{4}{r}\int_{\partial B_{1}}  u_r \langle\nabla u_r,y \rangle \d \mathcal{H}^{n-1}\,.
\end{align}
Plugging \eqref{e:H'+H} and \eqref{e:E'} in \eqref{e:Phi'} we get
\begin{align*}
\Phi'_{u}(r)
&= \dfrac{1}{r}\int_{\partial B_{1}} (| \nabla u_{r}|^{2} + 2u_{r} \,) \d \mathcal{H}^{n-1} -(n+2)\dfrac{\EEE_u(r)}{r^{n+3}}
+\dfrac{8}{r^{n+4}}\HHH_u(r)-\dfrac{4}{r}\int_{\partial B_{1}}  u_r \langle\nabla u_r,y \rangle \d \mathcal{H}^{n-1}
\notag\\& 
=-\frac{n+2}{r}\Phi_{u}(r)-\dfrac{2(n-2)}{r^{n+4}}\HHH_{u}(r)
+\frac{1}{r}\int_{\partial B_{1}} \left( | \nabla u_{r}|^{2} + 2u_{r} -4 u_{r}\langle\nabla u_{r},y \rangle\right) \d \mathcal{H}^{n-1}
\notag\\
 &=-\frac{n+2}{r}\Phi_{u}(r)+
 \frac{1}{r}\int_{\partial B_{1}} \Big( | \nabla u_{r}|^{2} + 2u_{r} -2(n-2)u_{r}^{2}-4 u_{r}\langle\nabla u_{r},y \rangle\Big) \d \mathcal{H}^{n-1}\notag\\
&=-\dfrac{n+2}{r}\Phi_{u}(r) +\dfrac{1}{r} \int_{\partial B_{1}} \Big( \big( \langle \nabla u_{r}, y \rangle  -2u_{r}\big)^{2} +| \nabla_{\tau} u_{r}|^{2} - 2nu_{r}^{2} + 2u_{r} \Big) \, \d \mathcal{H}^{n-1}\notag\\
&=\dfrac{n+2}{r} \left( \Phi_{w_{r}}(1) -\Phi_{u_{r}}(1) \right) + \dfrac{1}{r} \int_{\partial B_{1}} 
\left( \langle \nabla u_{r}, \nu \rangle  -2u_{r}\right)^{2} \d \mathcal{H}^{n-1}\,,
\end{align*}
where in the last equality we have used that
\begin{align*}
\Phi_{w_{r}}(1)= \dfrac{1}{n+2}\int_{\partial B_{1}} (| \nabla_{\tau} u_{r}|^{2} - 2nu_{r}^{2} + 2u_{r}) \, \d \mathcal{H}^{n-1}\,.
\end{align*} 
which follows from a direct computation of the energy of the $2$-homogeneous extension $w_{r}$.
\end{proof}
The second result is a simple consequence of the continuity assumptions on the coefficients.
\begin{lemma}\label{l:Dir vs A}
Assume (H4) and \eqref{e:normalization}. For every 
$v\in W^{1,2}(B_r)$ with $v\geq 0$ $\cL^n$-a.e. on $B_r$ then
\begin{align}
&\Big|\int_{B_1}|\nabla v_r|^2\d x-\int_{B_1}\langle\mathbb{A}(r x)\nabla v_r,\nabla v_r\rangle\d x\Big|
\leq\omega_{\mathbb{A}}(r)\int_{B_1}|\nabla v_r|^2\d x \label{e:Dir vs A}\\
&\Big|\int_{B_1}(1-f(r x))v_r\d x\Big|\leq\omega_f(r)\int_{B_1}v_r\d x\,,\label{e:1 vs f}
\end{align}
\end{lemma}
\begin{proof}
It suffices to take in to account the normalization assumption \eqref{e:normalization} and the definition of modulus of continuity.
\end{proof}
Let us first establish Weiss' quasi-monotonicity under the normalization condition.
\begin{proposition}\label{p:Weiss normalized}
Assume (H1), (H2), (H4) and \eqref{e:normalization}.
Let $u$ be the minimizer of $\mathscr{E}$ over $\mathbb{K}_{\psi,g}$.
There is a dimensional constant $C=C(n)>0$ such that if for some 
$\gamma\geq1$ 
\begin{equation}\label{e:bound EEE_u}
\|u_r\|_{L^\infty(B_2)}+\|\nabla u_r\|_{L^2(B_2;\R^n)}\leq \gamma
\end{equation}
for every $r\in(0,\frac12\mathrm{dist}(\zz,\partial\Omega))$, then 
\begin{equation}\label{e:Weiss normalized}
\frac{\d}{\d r}\left(\Phi_{u}(r)+C\gamma^2\int_0^r\frac{\omega(t)}t\d t\right)\geq 
\frac{1}{r}\int_{\partial B_1} (\langle\nabla u_r, x\rangle-2u_r)^2 
\d\mathcal{H}^{n-1},
\end{equation}
for every $r\in(0,\frac12\mathrm{dist}(\zz,\partial\Omega))$, where 
$\omega(r):=\omega_{\mathbb{A}}(r)+\omega_f(r)$. 

In particular, $\Phi_u$ has a finite right limit in $0$ denoted by $\Phi_u(0^+)$.
\end{proposition}
\begin{proof}
For $r\in(0,\mathrm{dist}(\zz,\partial\Omega))$, we use formula \eqref{e:Phi'} for $\Phi_u'$ in combination 
with one of the following alternatives
\begin{itemize}
\item[(a)] $\Phi_{w_r}(1)\geq \Phi_{u_r}(1)$,
\item[(b)] $\Phi_{w_r}(1)<\Phi_{u_r}(1)$.
\end{itemize}
In case (a) we conclude that 
\begin{equation}\label{e:qmon 1}
\Phi_u'(r)\geq \dfrac{1}{r} \int_{\partial B_{1}} 
\left( \langle \nabla u_{r}, \nu \rangle  -2u_{r}\right)^{2} \d \mathcal{H}^{n-1}\,.
\end{equation}
Otherwise, being $u_r|_{\partial B_1}=w_r|_{\partial B_1}$, the inequality defining case (b) rewrites as
\begin{equation}\label{e:qmon 2}
\EEE_{w_r}(1)<\EEE_{u_r}(1)\,.
\end{equation}
Thus we may estimate $\Phi_{w_r}(1)-\Phi_{u_r}(1)$ from below by taking into account Lemma~\ref{l:Dir vs A} 
and that $u_r$ minimizes the functional in \eqref{e:enrg2 u_r} with respect to its boundary values, to conclude that 
\begin{align*}
\Phi_{w_r}(1)&-\Phi_{u_r}(1)=\EEE_{w_r}(1)-\EEE_{u_r}(1)\notag\\
&\geq\int_{B_1}\left(\langle\mathbb{A}(rx)\nabla w_r,\nabla w_r\rangle+2f(rx)w_r\right)\d x
-\int_{B_1}\left(\langle\mathbb{A}(rx)\nabla u_r,\nabla u_r\rangle+2f(rx)u_r\right)\d x\notag\\
&-\omega(r)(\EEE_{w_r}(1)+\EEE_{u_r}(1))
\stackrel{\eqref{e:qmon 2}}{\geq}-2\omega(r)\EEE_{u_r}(1)
\stackrel{\eqref{e:bound EEE_u}}{\geq}-C\gamma^2 \omega(r)\,,
\end{align*}
where we have set $\omega(r)=\omega_{\mathbb{A}}(r)+\omega_f(r)$ and $C=C(n)>0$.
Hence, in case (b) 
we infer that 
\begin{equation}\label{e:qmon 2.2}
\Phi_u'(r)\geq -C\gamma^2\frac{\omega(r)}{r}+\dfrac{1}{r} \int_{\partial B_{1}} 
\left( \langle \nabla u_{r}, \nu \rangle  -2u_{r}\right)^{2} \d \mathcal{H}^{n-1}\,.
\end{equation}
Inequalities \eqref{e:qmon 1} and \eqref{e:qmon 2.2} provide \eqref{e:Weiss normalized} for every $r\in(0,\frac12\mathrm{dist}(\zz,\partial\Omega))$.
\end{proof}
\begin{remark}
Recalling that $f$ and $\mathbb{A}$ are Dini continuous by (H4), 
the modulus of continuity $\omega$ provided by Proposition~\ref{p:Weiss normalized} is in turn Dini continuous.
\end{remark}
\begin{remark}\label{r:Weiss generalized}
An inspection of the proof above shows that Weiss' formula can be deduced even for a weaker notion of Dini continuity, 
that is actually the one used in \cite[Theorem~2.1 and Remark~2.2]{Li} to infer the mentioned $C^1$ regularity of solutions. 
In this respect, we need a different version of Lemma~\ref{l:Dir vs A}. To this aim, thanks to the mentioned Schauder estimates, 
in place of \eqref{e:Dir vs A} and \eqref{e:1 vs f} we may consider for $r$ sufficiently small the inequalities
\begin{align}
&\Big|\int_{B_1}|\nabla u_r|^2\d x-\int_{B_1}\langle\mathbb{A}(r x)\nabla u_r,\nabla u_r\rangle\d x\Big|
\leq \widetilde{\omega}(r)\|\nabla u_r\|^2_{L^\infty(B_1;\R^n)}
\,,\label{e:A vs Id bis}\\
&\Big|\int_{B_1}(1-f(r x))u_r\d x\Big|\leq 
\widetilde{\omega}(r)\|u_r\|_{L^\infty(B_1)}\,,\label{e:1 vs f bis}
\end{align}
where
\[
\widetilde{\omega}(r):=\cL^n(B_1) \left(\sup_{y\in B_2}\fint_{B_r(y)}\Big(\|\mathbb{A}(x)-\mathbb{A}(y)\|^2
+|f(x)-f(y)|^2\Big)\d x\right)^{\sfrac12}\,.
\]
Note that $\widetilde{\omega}$ is not a modulus of continuity according to the definition given above as it is not increasing.
Despite this, assume that it satisfies \eqref{e:Dini continuity}. 
Then, on one hand \cite[Theorem~2.1 and Remark~2.2]{Li} provide $C^1$ regularity with a uniform modulus of continuity for the 
gradient of the solution; on the other hand \eqref{e:A vs Id bis} and 
\eqref{e:1 vs f bis} together with \eqref{e:bound E_u+H_u} below (rather than \eqref{e:bound EEE_u}) yield Weiss' quasi-monotonicity formula. As it will be discussed in Section~\ref{ss:WM qmon} below, it is not restrictive to assume \eqref{e:bound E_u+H_u} thanks exactly to \cite[Theorem~2.1 and Remark~2.2]{Li}. 

Finally, we note that even weaker notions of continuity are allowed to get $C^1$ regularity of solutions to divergence form elliptic equations (cf. \cite{DK17,DEK18}). In this respect, Reifenberg vanishing flatness 
of $\Reg(u)$ 
had already been proved in case of $\mathrm{VMO}$ coefficients in \cite{BlankHao15bis} by means of PDEs arguments rather than using the variational approach (cf. the introduction for the definition of $\Reg(u)$).
\end{remark}

For what Monneau's formula is concerned, let 
$v$ be any positive $2$-homogeneous polynomial solution of 
\begin{equation}\label{e:v}
 \Delta v = 1\quad \textrm{on $\R^n$}.
\end{equation}
Then by $2$-homogeneity, elementary calculations lead to 
 \begin{equation}\label{e:Phiv}
\Phi_v(r)=\Phi_v(1)= \int_{B_1} v\d y,
 \end{equation}
for all $r>0$. It is easy to prove that the value above is a dimensional constant independent of $v$, 
which we denote by $\theta$. Then, being the space of polynomials of degree $2$ finite dimensional, 
and being $v$ $2$-homogeneous we infer that 
\begin{equation}\label{e:estimate v norms}
\|\nabla v\|_{L^2(B_1)}+\|v\|_{L^2(\partial B_1)}\leq C(n)\,.
\end{equation}
We prove next a quasi-monotonicity formula for solutions of the obstacle problem in case $x_0\in\Gamma_u$ 
is a singular point of the free boundary, namely it is such that 
\begin{equation}\label{e:en BU sing Monneau}
 \Phi_u(0^+)=\theta\,. 
\end{equation}
To prove Monneau's formula we need to strengthen condition \eqref{e:bound EEE_u} (cf. \eqref{e:bound E_u+H_u} below).
\begin{proposition}\label{p:Monneau normalized}
Assume (H1), (H2), (H5) with $a=1$, and \eqref{e:normalization}.
Let $u$ be the minimizer of $\mathscr{E}$ over 
$\mathbb{K}_{\psi,g}$.
There exists a dimensional constant 
$C=C(n)>0$ such that if for some $\gamma\geq1$
\begin{equation}\label{e:bound E_u+H_u}
\|u_r\|_{L^\infty(B_2)}+\|\nabla u_r\|_{L^\infty(B_2;\R^n)}\leq \gamma
\end{equation}
for every $r\in(0,\frac12\mathrm{dist}(\zz,\partial\Omega))$, 
then the function 
 \begin{equation}\label{e:Monneau normalized}
(0,{\textstyle{\frac12}}\mathrm{dist}(\zz,\partial\Omega))
\ni r\longmapsto  
  \int_{\partial B_1} (u_r-v_r)^2\,\d x + 
C\gamma^2\int_0^r\frac{dt}{t}\int_0^t\frac{\omega(s)}{s}\d s
\end{equation}
is nondecreasing, where $v$ is any positive $2$-homogeneous polynomial solution of \eqref{e:v}, 
and $\omega$ is the modulus of continuity provided by Proposition~\ref{p:Weiss normalized}. 
 \end{proposition}
\begin{proof}
Let $w_r:=u_r-v$, then by scaling and by taking into account the $2$-homogeneity of $v$ we get
\begin{align}\label{e:monneau 1}
\frac{\d}{\d r}&\left(\frac1{r^{n+3}}\int_{\partial B_1}(u-v)^2\d\HH\right)=
\frac{\d}{\d r}\left(\int_{\partial B_1}w_r^2\d\HH\right)=\frac2r\int_{\partial B_1}w_r(\langle \nabla u_r,x\rangle-2u_r)\d\HH
\notag\\
&\geq\frac2r\int_{\partial B_1}w_r(\langle \mathbb{A}(rx)\nabla u_r,x\rangle-2u_r)\d\HH
-\frac{\omega_{\mathbb{A}}(r)}r\|w_r\|_{L^2(\partial B_1)}\|\nabla u_r\|_{L^2(\partial B_1)}\notag\\
&\geq\frac2r\int_{\partial B_1}w_r(\langle \mathbb{A}(rx)\nabla u_r,x\rangle-2u_r)\d\HH
-C\gamma^2\frac{\omega_{\mathbb{A}}(r)}r\,,
\end{align}
for some $C=C(n)>0$, where we have used \eqref{e:estimate v norms} 
and \eqref{e:bound E_u+H_u} in the last inequality.

We use next the integration by parts formula \eqref{e:singular variation} in 
Corollary~\ref{c:integration by parts} to get
\begin{align*}
\int_{\partial B_1}&w_r \langle \mathbb{A}(rx)\nabla u_r,x\rangle\d\HH=
\int_{B_1}\langle \mathbb{A}(rx)\nabla u_r,\nabla w_r\rangle\d x+
\int_{B_1}(f(rx)-\zeta(rx))w_r\d x:=I_1+I_2\,.
\end{align*}
We estimate the two addends above separately. We start off with $I_1$:
\begin{align}\label{e:I_1}
I_1&=\int_{B_1}\langle \mathbb{A}(rx)\nabla u_r,\nabla u_r\rangle\d x
-\int_{B_1}\langle \mathbb{A}(rx)\nabla u_r,\nabla v\rangle\d x
\notag\\
&\geq\Phi_{u_r}(1)-\omega_{\mathbb{A}}(r)\int_{B_1}|\nabla u_r|^2\d x
-2\int_{B_1}u_r\d x+2\HHH_{u_r}(1)
-\int_{B_1}\langle \mathbb{A}(rx)\nabla u_r,\nabla v\rangle\d x\,.
\end{align}
By taking advantage of the $2$-homogeneity of $v$ and that $\triangle v=1$ (cf. \eqref{e:v}) we get that
\begin{align*}
-\int_{B_1}\langle \mathbb{A}(rx)\nabla u_r,\nabla v\rangle\d x&
\geq-\int_{B_1}\langle \nabla u_r,\nabla v\rangle\d x
-\omega_{\mathbb{A}}(r)\|\nabla u_r\|_{L^2(B_1)}\|\nabla v\|_{L^2(B_1)}\notag\\
&=\int_{B_1}u_r\d x-2\int_{\partial B_1}u_rv\d\HH
-\omega_{\mathbb{A}}(r)\|\nabla u_r\|_{L^2(B_1)}\|\nabla v\|_{L^2(B_1)}\notag\\
&\geq\int_{B_1}u_r\d x-2\int_{\partial B_1}u_rv\d\HH
-C\gamma\omega_{\mathbb{A}}(r)\,, 
\end{align*}
for some $C=C(n)>0$, where we have used \eqref{e:estimate v norms} and 
\eqref{e:bound E_u+H_u} in the last inequality. Plugging the latter estimate in \eqref{e:I_1} we infer that 
\begin{align}\label{e:I_12}
I_1&\geq\Phi_{u_r}(1)
-\int_{B_1}u_r\d x+
2\int_{\partial B_1}u_rw_r\d\HH
-C\gamma^2\omega_{\mathbb{A}}(r)\,. 
\end{align}
Next note that by \eqref{e:zeta} in Proposition~\ref{p:PDE u}
\begin{align}\label{e:I_2}
I_2&=\int_{B_1}(f(rx)-\zeta(rx))u_r\d x-\int_{B_1}(f(rx)-\zeta(rx))v\d x\notag\\
&\geq\int_{B_1}(f(rx)-f(rx)\chi_{\{u_r=0\}})u_r\d x-\int_{B_1}f(rx)v\d x\notag\\
&=\int_{B_1}f(rx)\chi_{\{u_r>0\}}u_r-\int_{B_1}f(rx)v\d x=\int_{B_1}f(rx)(u_r-v)\d x\,.
\end{align}
We now use \eqref{e:Phiv}, \eqref{e:I_12} and \eqref{e:I_2}, to estimate \eqref{e:monneau 1} 
\begin{align}\label{e:monneau 2}
\frac{\d}{\d r}&\left(\frac1{r^{n+3}}\int_{\partial B_1}(u-v)^2\d\HH\right)\geq
\frac1r(\Phi_{u_r}(1)-\Phi_{v}(1))\notag\\
&+\frac1r\int_{B_1}(1-f(rx))(v-u_r)\d x
-C\gamma^2\frac{\omega_{\mathbb{A}}(r)}r\notag\\ 
&\stackrel{\eqref{e:estimate v norms}}{\geq}\frac1r(\Phi_{u_r}(1)-\Phi_{u}(0^+))
-C\gamma^2\frac{\omega(r)}r\,, 
\end{align}
for some $C=C(n)>0$. Therefore, we may finally use Proposition~\ref{p:Weiss normalized} to conclude that
\begin{align*}
\frac{\d}{\d r}&\left(\frac1{r^{n+3}}\int_{\partial B_1}(u-v)^2\d\HH\right)\\&\geq
\frac1r\int_0^r\frac{1}{t}\int_{\partial B_1} (\langle\nabla u_t, x\rangle-2u_t)^2 d\mathcal{H}^{n-1}
-\frac Cr\gamma^2\int_0^r\frac{\omega(t)}t\d t
-C\gamma^2\frac{\omega(r)}r\,.
\end{align*}
The conclusion then follows at once.
\end{proof}

\subsection{Weiss' and Monneau's quasi-monotonicity: general case}\label{ss:WM qmon}

To establish Weiss' and Monneau's monotonicity in general, we follow \cite{FoGeSp15} and show 
that by means of a change of variables one can always reduce to the normalized setting in \eqref{e:normalization}
for every free boundary point $x_0\in\Gamma_u$. 
Moreover, the new quantities appearing in the corresponding obstacle problems under such a transformation satisfy assumptions (H1), (H2) and either (H3) or (H4) or (H5), with uniform constants with respect 
to $x_0$, according to the assumption imposed on $\mathbb{A}$ 
and $f$.

Indeed, let $x_0 \in \Gamma_u$ be any point of the free boundary, consider the affine change of variables
\[
x\longmapsto x_0+f^{-\sfrac12}(x_0)\mathbb{A}^{\sfrac12}(x_0)x =: x_0 + \LL(x_0)\, x
\]
Changing variables leads to
\begin{equation}\label{e:cambio di coordinate3}
\mathscr{E}(u)=f^{1-\frac{n}{2}}(x_0)\det(\mathbb{A}^{\sfrac12}(x_0))\,\mathscr{E}_{\LL(x_0)}(u_{\LL(x_0)}),
\end{equation}
where we have set 
$\Omega_{\LL(x_0)}:=\LL^{-1}(x_0)\,(\Omega-x_0)$, and 
\begin{equation}\label{e:enrgA}
\mathscr{E}_{\LL(x_0)}(v):=\int_{\Omega_{\LL(x_0)}}\left(
\langle {\mathbb C}_{\LL(x_0)}(x)\nabla v,\nabla v\rangle 
+ 2\frac{f_{\LL(x_0)}}{f(x_0)}\,v\right)\d x,
\end{equation}
with
\begin{align}
&u_{{\LL(x_0)}}(x)  :=u\big(x_0+\LL(x_0)x\big), 
\label{e:cambio di coordinate1}\\
&f_{{\LL(x_0)}}(x)  :=f\big(x_0+\LL(x_0)x\big), \notag
\notag\\
&{\mathbb C}_{\LL(x_0)}(x)  :=
\mathbb{A}^{-\sfrac12}(x_0)\mathbb{A}(x_0+\LL(x_0)x)\mathbb{A}^{-\sfrac12}(x_0).\notag
\end{align}
Note that $f_{\LL(x_0)}(\underline{0})=f(x_0)$ and ${\mathbb C}_{\LL(x_0)}(\underline{0})=\mathrm{Id}$.
Moreover, the free boundary is transformed under this map into
\[
\Gamma_{u_{\LL(x_0)}}=\LL^{-1}(x_0)(\Gamma_u-x_0),
\] 
and the energy $\mathscr{E}$ in \eqref{e:enrg} is minimized by $u$ on $\mathbb{K}_{\psi,g}$ if and only if 
$\mathscr{E}_{\LL(x_0)}$ in \eqref{e:enrgA} is minimized by $u_{\LL(x_0)}$ in \eqref{e:cambio di coordinate1} on 
$\mathbb{K}_{\psi(\LL^{-1}(x_0)(\cdot-x_0),g(\LL^{-1}(x_0)(\cdot-x_0))}$. In particular, the normalization assumption \eqref{e:normalization}
is satisfied. 

Moreover, ${\mathbb C}_{\LL(x_0)}$ and $\frac{f_{\LL(x_0)}}{f(x_0)}$ satisfy (H1)-(H5) with uniform constants. 
Indeed, for what (H1) is concerned, it is clear that 
$\mathbb{C}_{\LL(x_0)}(\cdot)$ is symmetric, bounded and coercive, with
  \begin{equation}\label{e:H1 non norm}
   \Lambda^{-2}|\xi|^2\leq \langle \mathbb{C}_{\LL(x_0)}(x)\xi, \xi\rangle \leq \Lambda^2|\xi|^2
  \end{equation}
 for $\cL^n$-a.e. $x\in\Omega_{\LL(x_0)}$, and for every 
 $\xi\in \R^n$.  Note that
 \begin{equation}\label{e:H2 non norm}
 \frac{c_0}{\|f\|_{L^\infty(\Omega)}}<\frac{f_{\LL(x_0)}}{f(x_0)}\leq\frac{\|f\|_{L^\infty(\Omega)}}{c_0}
 \end{equation}
 for $\cL^n$-a.e. $x\in\Omega_{\LL(x_0)}$, so that (H2) holds.
 Moreover, on setting $\overline{\omega}_{\mathbb{A}}(t):=
 (n\Lambda)^2\omega_{\mathbb{A}}\left({\textstyle\sqrt{\frac{n\Lambda}{c_0}}}t\right)$, it is clear that for every $x$, $y\in\Omega_{\LL(x_0)}$
\begin{equation}\label{e:H4 non norm}
\|{\mathbb C}_{\LL(x_0)}(x)-{\mathbb C}_{\LL(x_0)}(y)\|\leq \overline{\omega}_{\mathbb{A}}(|x-y|)\,.
\end{equation}
Analogously, $\frac{f_{\LL(x_0)}}{f(x_0)}$ has modulus of continuity
 $\overline{\omega}_f(t):=c_0^{-1}\omega_f\left({\textstyle\sqrt{\frac{n\Lambda}{c_0}}}t\right)$.
 Therefore, either (H3) or (H4) or (H5) holds, according to the corresponding assumption on $\mathbb{A}$ and $f$. 

Furthermore, we note that in view of \eqref{e:H1 non norm} and 
\eqref{e:H2 non norm} formula \eqref{e:bound EEE_u} is 
satisfied uniformly in $x_0$ and $r$.
More precisely, Remark~\ref{r:Caccioppoli} yields that if $R>0$,
$K\subset\Omega$ is compact and $x_0\in K$, then for every 
 $r\in \big(0, \frac{1}{4R}\mathrm{dist}(K, \partial\Omega)\big)$,
and for some constant $C=C(n,c_0,\Lambda,\|f\|_{L^\infty},R,K)>0$ 
it is true that
 \begin{equation}\label{e:u_r limitata C1alfa 2}
  \|u_{\LL(x_0),r}\|_{L^\infty(B_R)}  
  +\|\nabla u_{\LL(x_0),r}\|_{L^2(B_R;\R^n)}\leq C\,,
 \end{equation}
where we have set $u_{\LL(x_0),r}:=(u_{\LL(x_0)})_{\zz,r}$ (notice that
$\nabla u_{\LL(x_0),r}(\cdot)=\LL^t(x_0)\nabla u(x_0+\LL(x_0)\cdot)$).

We are now ready to establish Weiss' quasi-monotonicity formula by
applying Proposition~\ref{p:Weiss normalized} to 
$u_{\LL(x_0)}$ thanks to the discussion above.
\begin{theorem}[Weiss' quasi-monotonicity formula]\label{t:Weiss} 
Assume (H1), (H2) and (H4). 
Let $u$ be the minimizer of $\mathscr{E}$ over 
$\mathbb{K}_{\psi,g}$.
If $K\subset\Omega$ is a compact set, there is a constant $C=C(n,c_0,\Lambda,\|f\|_{L^\infty},K)>0$ such that for all $x_0\in K\cap \Gamma_u$ 
\begin{equation}\label{e:Weiss}
\frac{\d}{\d r}\Big(\Phi_{u_{\LL(x_0)}}(r) + C\int_0^r\frac{\overline{\omega}(t)}{t}\d t\Big)\geq 
\frac{1}r
\int_{\partial B_1} (\langle\nabla u_{\LL(x_0),r}, x\rangle-2u_{\LL(x_0),r})^2 \d\mathcal{H}^{n-1},
\end{equation}
for every $r\in(0,\frac14\mathrm{dist}(K,\partial\Omega))$,
where $\overline{\omega}(r):=\overline{\omega}_{\mathbb{A}}(r)+\overline{\omega}_f(r)$. 

In particular, $\Phi_{u_{\LL(x_0)}}$ has finite right limit 
$\Phi_{u_{\LL(x_0)}}(0^+)$ in zero, 
and for all $r\in(0,\frac14\mathrm{dist}(K,\partial\Omega))$,
\begin{equation}\label{e:Phi(r)-Phi(0)}
 \Phi_{u_{\LL(x_0)}}(r)-\Phi_{u_{\LL(x_0)}}(0^+)\geq -C\int_0^r\frac{\overline{\omega}(t)}{t}\d t\,.
\end{equation} 
\end{theorem}
For what concerns Monneau's quasi-monotonicity formula we may apply Proposition~\ref{p:Monneau normalized} to $u_{\LL(x_0)}$ 
on condition that \eqref{e:bound E_u+H_u} is satisfied. 
This follows from the discussion above, and thanks to \cite[Theorem~2.1 and Remark~2.2]{Li} which provide a modulus of continuity for $\nabla u_{\LL(x_0)}$ depending only on $n$, $c_0$, $\Lambda$, 
$\overline{\omega}_{\mathbb{A}}$, and $\|f\|_{L^\infty(\Omega)}$ 
(cf. \eqref{e:H1 non norm}-\eqref{e:H4 non norm}).
\begin{theorem}[Monneau's quasi-monotonicity formula]\label{t:Monneau}
Assume (H1), (H2) and (H5) with $a=1$. 
Let $u$ be the minimizer of $\mathscr{E}$ over 
$\mathbb{K}_{\psi,g}$. If $K\subset\Omega$ is a compact set and 
\eqref{e:Phiv} holds for $x_0\in K\cap \Gamma_u$, then there exists a constant $C=C(n,c_0,\Lambda,\|f\|_{L^\infty},K)>0$ such that the function 
 \begin{align}\label{e:Monneau}
  \big(0,{\textstyle{\frac14}}\mathrm{dist}(K,\partial\Omega)\big)\ni r\longmapsto   
  \int_{\partial B_1} ( u_{\LL(x_0),r}-v)^2\,\d x
  + C\gamma^2\int_0^r\frac{dt}{t}\int_0^t\frac{\overline{\omega}(s)}{s}\d s\,.
\end{align}
is nondecreasing, where $v$ is any $2$-homogeneous polynomial solution of \eqref{e:v}, and $\overline{\omega}$ 
is the modulus of continuity provided by Theorem~\ref{t:Weiss}. 
\end{theorem}

\section{Free boundary analysis}\label{s:applications}

The regularity of the free boundary of the minimizer $u$ of 
$\mathscr{E}$ on $\mathbb{K}_{\psi,g}$ can be established 
thanks to the Weiss' and Monneau's quasi-monotonicity formulas proved in Section~\ref{s:q-mon formula} at least if assumption 
(H3) is satisfied.
In doing this we follow the approach introduced in \cite{Weiss,Monneau03} for the classical obstacle problem related to the Dirichlet energy, and developed in \cite{FoGeSp15,Ger17, FoGerSp17, FoGerSp20} both for linear elliptic operators in divergence form 
and in the nonlinear setting with suitable smoothness assumptions
(see also \cite{PSU} for a systematic presentation).

In particular, in this section we improve upon \cite[Theorems~4.12 and 4.14]{FoGeSp15}, \cite[Theorem~1.3]{Ger17} and 
\cite[Theorem~2.1]{FoGerSp20}, since in all those cases the matrix field $\mathbb{A}$ turns out to be in particular H\"older continuous due to Sobolev type embeddings. 

In the ensuing proof we will highlight only the substantial changes since the arguments are essentially those given in 
\cite{FoGeSp15}.

\begin{proof}[Proof of Theorem~\ref{t:linear}]
First recall that we may establish the conclusions for the function $u=w-\psi$ introduced in Section~\ref{s:prel}.
Given this, the only minor change to be done to the arguments in \cite[Section~4]{FoGeSp15} is related to the 
freezing of the energy where the regularity of the coefficients plays a substantial role. 
More precisely, under assumption \eqref{e:normalization} we have in view of Lemma~\ref{l:Dir vs A}
 \[
 \left|\int_{B_1}\big(\mathbb{A}(rx)\nabla v, \nabla v\rangle + 2f(rx)v\big)\d x - 
 \EEE_v(1)\right|\leq \omega(r)\EEE_v(1)
 \]
 for all $v\in W^{1,2}(B_1)$.
 
We then describe shortly how to infer all the conclusions. We start off recalling that the 
quasi-monotonicity formulas established in \cite[Section~3]{FoGeSp15} 
are to be substituted by those in Section~\ref{s:q-mon formula}.
Then the $2$-homogeneity of blow up limits in \cite[Proposition~4.2]{FoGeSp15} now follows from Theorem~\ref{t:Weiss}.
Nondegeneracy of blow up limits is contained in Proposition~\ref{p:quadratic growth detachment} (see \cite[Lemma~4.3]{FoGeSp15}).
The classification of blow up limits is obtained exactly as in \cite[Proposition~4.5]{FoGeSp15}.
Uniqueness of blow up limits at regular points, that follows from \cite[Lemma~4.8]{FoGeSp15}, can be obtained with essentially 
no difference. The proofs of \cite[Propositions~4.10,~4.11, Theorems~4.12,~4.14]{FoGeSp15} remain unchanged. 
\end{proof}
\begin{remark}\label{r:extensions}
Thanks to the quasi-monotonicity formulas in Section~\ref{s:q-mon formula} we expect to be possible to deduce results 
analogous to Theorem~\ref{t:linear} under assumption (H5) with $a>2$ (cf. \cite[Theorem~2.1]{FoGerSp20} for such a statement 
if $\mathbb{A}\in W^{1,p}(\Omega;\R^{n\times n})$, $p>n$, and $f$ satisfies (H5) with $a>2$). 
Moreover, in view of Remark~\ref{r:Weiss generalized} it is likely that analogous results hold even in case $\widetilde{\omega}$ 
there satisfies a double-Dini continuity condition. 
We do not insist on this issue here 
since several arguments should be carefully checked along the proofs of  \cite{FoGeSp15}. 
\end{remark}

%
%

\bibliographystyle{plain}

\end{document}